\documentclass[reqno,12pt]{amsart}
\usepackage{amsfonts}
\usepackage{bbm}
\usepackage{} 
\numberwithin{equation}{section}
\usepackage{indentfirst}
\usepackage{color}
\usepackage{amssymb}
\usepackage{mathrsfs}
\usepackage[shortcuts]{extdash}
\usepackage{xy}
\xyoption{all}

\theoremstyle{plain} 
\newtheorem{theorem}{\bf Theorem}[section]
\newtheorem{lemma}[theorem]{\bf Lemma}
\newtheorem{corollary}[theorem]{\bf Corollary}
\newtheorem{proposition}[theorem]{\bf Proposition}

\theoremstyle{definition} 
\newtheorem{definition}[theorem]{\bf Definition}
\newtheorem{condition}[theorem]{\bf Condition}
\newtheorem{remark}[theorem]{\bf Remark}
\newtheorem{example}[theorem]{\bf Example}

\newcommand{\bt}{\begin{theorem}}
\newcommand{\et}{\end{theorem}}
\newcommand{\bl}{\begin{lemma}}
\newcommand{\el}{\end{lemma}}
\newcommand{\bd}{\begin{definition}}
\newcommand{\ed}{\end{definition}}
\newcommand{\bc}{\begin{corollary}}
\newcommand{\ec}{\end{corollary}}
\newcommand{\bp}{\begin{proof}}
\newcommand{\ep}{\end{proof}}
\newcommand{\bx}{\begin{example}}
\newcommand{\ex}{\end{example}}
\newcommand{\br}{\begin{remark}}
\newcommand{\er}{\end{remark}}
\newcommand{\be}{\begin{equation}}
\newcommand{\ee}{\end{equation}}
\newcommand{\ba}{\begin{align}}
\newcommand{\ea}{\end{align}}
\newcommand{\bn}{\begin{enumerate}}
\newcommand{\en}{\end{enumerate}}
\newcommand{\bcs}{\begin{cases}}
\newcommand{\ecs}{\end{cases}}

\makeatletter
\renewcommand{\section}{\@startsection{section}{1}{0mm}
  {-\baselineskip}{0.5\baselineskip}{\bf\leftline}}
\makeatother

\begin{document}

\title[A bijection between support $\tau$-tilting subcategories and $\tau$-cotorsion pairs]{A bijection between support $\tau$-tilting subcategories and $\tau$-cotorsion pairs in extriangulated categories} 

\author[Z. Zhu, J. Wei]{Zhiwei Zhu, Jiaqun Wei$^\ast$}

\address{Institute of Mathematics, School of Mathematical Sciences, Nanjing Normal University, Nanjing, 210023, P.~R.~China}
\email{1985933219@qq.com (Zhu)}

\address{Department of Mathematics, Zhejiang Normal University, Jinhua, 321004, Zhejiang, P.~R.~China}
\email{weijiaqun5479@zjnu.edu.cn (Wei)}


\thanks{$*$~Corresponding author.}


\begin{abstract}
Let $\mathscr{C}$ be an extriangulated category with enough projectives and injectives. We give a new definition of tilting subcategories of $\mathscr{C}$ and prove it coincides with the definition given in \cite{7}. As applications, we introduce the notions of support $\tau$-tilting subcategories and $\tau$-cotorsion pairs of $\mathscr{C}$. We build a bijection between support $\tau$-tilting subcategories and certain $\tau$-cotorsion pairs. Moreover, this bijection induces a bijection between tilting subcategories and certain cotorsion pairs.
\end{abstract}

\maketitle

\section{Introduction}
Tilting theory is a generation of Morita equivalences and plays an important role in the representation theory of algebra. It originated with the study of reflection functors in \cite{2,1}. The first set of axioms for a titling module was described by Brenner and Butler in \cite{3}. After that, tilted algebras \cite{4} were defined by Happel and Ringel as endomorphism algebras of tilting modules over hereditary algebras. Nowadays, tilting theory has been generalized in many directions. Krause \cite{5} defined tilting objects in exact categories and Sauter \cite{6} defined tilting subcategories in exact categories. Recently, Zhu and Zhuang \cite{7} generalized the definition of tilting subcategories in extriangulated categories.

Extriangulated categories were introduced by Nakaoka and Palu \cite{8}, which share some properties of triangulated categories and exact categories. There are some examples of extriangulated categories such as exact categories and extension-closed subcategories of triangulated categories, while some extriangulated categories may be neither exact nor triangulated categories. Hence, some results of triangulated categories and exact categories can be generalized to extriangulated categories.

Adachi, Iyama and Reiten \cite{9} introduced $\tau$-tilting theory on finite dimensional algebra. One can refer to \cite{10} for more details. Then Iyama, Jørgensen and Yang \cite{11} introduced support $\tau$-tilting subcategories in functor categories. Later on, support $\tau$-tilting subcategories were generalized in Hom-finite abelian categories with enough projectives by Liu and Zhou \cite{12}. Asadollahi, Sadeghi and Treffinger \cite{13} showed that there is a bijection between support $\tau$-tilting subcategories and $\tau$-cotorsion triples, which is a generalization of \cite{14}. Recently, Pan, Zhang and Zhu \cite{15} defined support $\tau$-tilting subcategories in exact categories and gave a bijection between support $\tau$-tilting subcategories and certain $\tau$-cotorsion pairs.

The main result of this paper is to provide a bijection between support $\tau$-tilting subcategories and certain $\tau$-cotorsion pairs in extriangulated categories. Moreover, this bijection induces a bijection between tilting subcategories and certain cotorsion pairs.

The paper is organized as follows. In Section 2, we summarize some definitions and results of extriangulated categories. In Section 3, we give a new definition of tilting subcategories in extriangulated categories and prove it coincides with the definition given in \cite{7}. Later on, we define support $\tau$-tilting subcategories and $\tau$-cotorsion pairs of extriangulated categories with enough projectives and injectives. Finally, we prove the main result.

\section{Preliminaries}
Throughout the article, we assume, unless otherwise stated, that $\mathscr{C}$ denotes an additive category, which is skeletally small and Krull-Schmidt. All subcategories
considered are full and closed under isomorphisms. We denote by $\mathscr{C}(A,B)$ the set of morphisms from $A$ to $B$ in $\mathscr{C}$. The composition of $a\in\mathscr{C}(A,B)$ and $b\in\mathscr{C}(B,C)$ is denoted by $ba$. For a subcategory $\mathscr{A}$ of $\mathscr{C}$, $a\in\mathscr{C}(A,C)$ is a {\em right} $\mathscr{A}$-{\em approximation} for $C\in\mathscr{C}$ if $A\in\mathscr{A}$ and $\mathscr{C}(A',a)$ is surjective for any $A'\in\mathscr{A}$. Dually, we can define {\em left $\mathscr{A}$-approximation}.

\subsection{Extriangulated categories}
Let us recall some notions concerning extriangulated categories from \cite{8}.

Let $\mathbb{E}:\mathscr{C}^{op}\times\mathscr{C}\rightarrow Ab$ be a biadditive functor, where $Ab$ is the category of abelian groups. For any pair of objects $A,C\in\mathscr{C}$, an element $\delta\in\mathbb{E}(C,A)$ is called an
$\mathbb{E}$-{\em extension}. The zero element $0\in\mathbb{E}(C,A)$ is called the {\em split} $\mathbb{E}$-{\em extension}. For any morphism $a\in\mathscr{C}(A,A')$ and $c\in\mathscr{C}(C',C)$, we have the following $\mathbb{E}$-extensions
$$\mathbb{E}(C,a)(\delta)\in\mathbb{E}(C,A'), \,\mathbb{E}(c,A)(\delta)\in\mathbb{E}(C',A),$$
which are denoted by $a_*\delta$ and $c^*\delta$, respectively.
\begin{definition}\label{morphism}
   \cite[Definition 2.3]{8} A morphism $(a,c):\delta\rightarrow \delta'$ of $\mathbb{E}$-extensions $\delta\in\mathbb{E}(C,A)$, $\delta'\in\mathbb{E}(C',A')$ is a pair of morphisms $a\in\mathscr{C}(A,A')$ and $c\in\mathscr{C}(C,C')$ satisfying  $a_*\delta=c^*\delta'$.
\end{definition}
Two sequences of morphisms $A\stackrel{x}{\longrightarrow}B\stackrel{y}{\longrightarrow}C$ and $A\stackrel{x'}{\longrightarrow}B'\stackrel{y'}{\longrightarrow}C$ in $\mathscr{C}$ are said to be {\em equivalent} if there exists an isomorphism $b\in\mathscr{C}(B,B')$ such that the following diagram is commutative.
$$\xymatrix{A \ar[r]^x \ar@{=}[d]& B\ar[r]^y \ar[d]^b_{\simeq}&C\ar@{=}[d]\\
A\ar[r]^{x'}&B'\ar[r]^{y'}&C}$$
We denote the equivalence class of $A\stackrel{x}{\longrightarrow}B\stackrel{y}{\longrightarrow}C$ by $[A\stackrel{x}{\longrightarrow}B\stackrel{y}{\longrightarrow}C]$. In addition, for any $A,C\in\mathscr{C}$, we denote as
$$0=[\xymatrix{A\ar[r]^{\tiny{\begin{pmatrix}1\\0\end{pmatrix}}\quad\quad}&A\oplus B\ar[r]^{\quad\tiny{\begin{pmatrix}0 & 1\end{pmatrix}}}&C}].
$$
For any two equivalence classes $[A\stackrel{x}{\longrightarrow}B\stackrel{y}{\longrightarrow}C]$ and $[A'\stackrel{x'}{\longrightarrow}B'\stackrel{y'}{\longrightarrow}C']$ we denote as
$$[A\stackrel{x}{\longrightarrow}B\stackrel{y}{\longrightarrow}C]\oplus [A'\stackrel{x'}{\longrightarrow}B'\stackrel{y'}{\longrightarrow}C'] = [A\oplus A'\stackrel{\tiny{\begin{pmatrix}
                    x & 0 \\
                    0 & x'
                  \end{pmatrix}}}{\longrightarrow}B\oplus B'\stackrel{\tiny{\begin{pmatrix}
                                      y & 0 \\
                                      0 & y'
                                    \end{pmatrix}}}{\longrightarrow}C\oplus C'].$$
\begin{definition}
  \cite[Definition 2.9]{8} Let $\mathfrak{s}$ be a correspondence, which associates an equivalence class $\mathfrak{s}(\delta)=[A\stackrel{x}{\longrightarrow}B\stackrel{y}{\longrightarrow}C]$ to each $\mathbb{E}$-extension $\delta\in\mathbb{E}(C,A)$. This $\mathfrak{s}$ is called a {\em realization} of $\mathbb{E}$ if for any morphism $(a,c):\delta\rightarrow \delta'$ with $\mathfrak{s}(\delta)=[A\stackrel{x}{\longrightarrow}B\stackrel{y}{\longrightarrow}C]$ and $\mathfrak{s}(\delta')=[A'\stackrel{x'}{\longrightarrow}B'\stackrel{y'}{\longrightarrow}C']$, there is a commutative diagram as follows:
  $$\xymatrix{A \ar[r]^x \ar[d]^a& B\ar[r]^y \ar[d]^b&C\ar[d]^c\\
  A'\ar[r]^{x'}&B'\ar[r]^{y'}&C'}
  $$
  A realization $\mathfrak{s}$ of $\mathbb{E}$ is said to be {\em additive} if the following conditions are satisfied:

  \;(a) For any $A,C\in\mathscr{C}$, the split $\mathbb{E}$-extension $0\in\mathbb{E}(C,A)$ satisfies $\mathfrak{s}(0)=0$.

  \;(b) $\mathfrak{s}(\delta\oplus \delta')=\mathfrak{s}(\delta)\oplus\mathfrak{s}(\delta')$ for any pair of $\mathbb{E}$-extensions $\delta$ and $\delta'$.
\end{definition}

\begin{definition}
  \cite[Definition 2.12]{8} We call the triple $\mathscr{C}=(\mathscr{C},\mathbb{E},\mathfrak{s})$ an extriangulated category if it satisfies the following conditions:

  ${\rm(ET1)}$ $\mathbb{E}:\mathscr{C}^{op}\times\mathscr{C}\rightarrow Ab$ is a biadditive functor.

  ${\rm(ET2)}$ $\mathfrak{s}$ is an additive realization of $\mathbb{E}$.

  ${\rm(ET3)}$ Let $\delta\in\mathbb{E}(C,A)$ and $\delta'\in\mathbb{E}(C',A')$ be any pair of $\mathbb{E}$-extensions, with
  $$\mathfrak{s}(\delta)=[A\stackrel{x}{\longrightarrow}B\stackrel{y}{\longrightarrow}C] \;{\rm and}\;
  \mathfrak{s}(\delta')=[A'\stackrel{x'}{\longrightarrow}B'\stackrel{y'}{\longrightarrow}C'].$$

  \quad\quad\quad For any commutative diagram
  $$\xymatrix{A \ar[r]^x \ar[d]^a& B\ar[r]^y \ar[d]^b&C\\
  A'\ar[r]^{x'}&B'\ar[r]^{y'}&C'}
  $$

  \quad\quad\quad in $\mathscr{C}$, there is a morphism $(a,c):\delta\rightarrow \delta'$ satisfying $cy=y'b$.

  ${\rm(ET3)^{op}}$ Dual of ${\rm(ET3)}$

  ${\rm(ET4)}$ Let $\delta\in\mathbb{E}(D,A)$ and $\delta'\in\mathbb{E}(F,B)$ be $\mathbb{E}$-extensions realized by $A\stackrel{f}{\longrightarrow}B\stackrel{f'}{\longrightarrow}D$

  \quad\quad\quad and $B\stackrel{g}{\longrightarrow}C\stackrel{g'}{\longrightarrow}F$, respectively. Then there exist an object $E\in\mathscr{C}$, a

  \quad\quad\quad commutative diagram
  $$\xymatrix{A \ar[r]^f \ar@{=}[d]& B\ar[r]^{f'}\ar[d]^g&D\ar[d]^d\\
  A\ar[r]^{h}&C\ar[r]^{h'}\ar[d]^{g'}&E\ar[d]^e\\&F\ar@{=}[r]&F}
  $$

  \quad\quad\quad in $\mathscr{C}$, and an $\mathbb{E}$-extension $\delta''\in\mathbb{E}(E,A)$ realized by $A\stackrel{h}{\longrightarrow}C\stackrel{h'}{\longrightarrow}E$, which

 \quad\quad\quad satisfy the following compatibilities:

  \quad\quad\quad (i) $D\stackrel{d}{\longrightarrow}E\stackrel{e}{\longrightarrow}F$ realizes $\mathbb{E}(F,f')(\delta')$,

  \quad\quad\quad (ii) $\mathbb{E}(d,A)(\delta'')=\delta$,

  \quad\quad\quad (iii) $\mathbb{E}(E,f)(\delta'')=\mathbb{E}(e,B)(\delta')$.

  ${\rm(ET4)^{op}}$ Dual of ${\rm(ET4)}$
\end{definition}

\begin{remark}
  (a) A sequence $A\stackrel{x}{\longrightarrow}B\stackrel{y}{\longrightarrow}C$ is called a {\em conflation} if it realizes some $\mathbb{E}$-extension $\delta\in\mathbb{E}(C,A)$. Then $x$ is called an {\em inflation} and $y$ is called a {\em deflation}. We say $A\stackrel{x}{\longrightarrow}B\stackrel{y}{\longrightarrow}C\stackrel{\delta}{\dashrightarrow}$ is an $\mathbb{E}$-{\em triangle}.

  (b) For a given $\mathbb{E}$-triangle $A\stackrel{x}{\longrightarrow}B\stackrel{y}{\longrightarrow}C\stackrel{\delta}{\dashrightarrow}$, we denote $A={\rm cocone}(y)$ and $C={\rm cone}(x)$. An $\mathbb{E}$-triangle is {\em split} if it realizes $0$.

  (c) A subcategory $\mathcal{U}$ of $\mathscr{C}$ is {\em closed under extensions} if for any conflation $A\stackrel{x}{\longrightarrow}B\stackrel{y}{\longrightarrow}C$ with $A,C\in\mathcal{U}$, we have $B\in\mathcal{U}$.

  (d) An object $P$ in $\mathscr{C}$ is {\em projective} if for any conflation $A\stackrel{x}{\longrightarrow}B\stackrel{y}{\longrightarrow}C$, $\mathscr{C}(P,y)$ is surjective. We denote the subcategory of projective objects by $\mathcal{P}(\mathscr{C})$. Dually, we can define {\em injective} objects and the subcategory of injective objects denoted by $\mathcal{I}(\mathscr{C})$. We say that $\mathscr{C}$ {\em has enough projectives} if for any $A\in\mathscr{C}$, there is a deflation $P\rightarrow A$ for some $P\in\mathcal{P}(\mathscr{C})$. Dually, we define that $\mathscr{C}$ {\em has enough injectives}.
\end{remark}

Throughout the article, we assume $\mathscr{C}$ has enough projectives and injectives.

\begin{definition}\label{2.5}
  Let $\mathscr{C}$ be an extriangulated category and $\mathcal{U}$ be a subcategory of $\mathscr{C}$. Then we define ${\rm Defl}(\mathcal{U})$ and ${\rm Infl}(\mathcal{U})$ as $${\rm Defl}(\mathcal{U})=\{F\in\mathscr{C}\;|\;\exists\;{\rm a}\;{\rm deflation}\; U\stackrel{g}{\longrightarrow}F\;{\rm for}\;{\rm some}\;U\in\mathcal{U}\}$$ and $${\rm Infl}(\mathcal{U})=\{S\in\mathscr{C}\;|\;\exists\;{\rm an}\;{\rm inflation}\; S\stackrel{f}{\longrightarrow}U\;{\rm for}\;{\rm some}\;U\in\mathcal{U}\}.$$
  We say $\mathcal{U}$ is {\em closed under deflations} (resp. {\em inflations}), if ${\rm Defl}(\mathcal{U})=\mathcal{U}$ (resp. ${\rm Infl}(\mathcal{U})=\mathcal{U}$).
\end{definition}

\subsection{Exact sequences in extriangulated categories}
In this section, we always assume that $\mathscr{C}$ is an extriangulated category satisfying the following conditions.

\begin{condition}\label{2.6}
  (WIC). \cite[Condition 5.8]{8}
  (1) For any pair of morphisms $f:X\rightarrow Y$ and $g:Y\rightarrow Z$ in $\mathscr{C}$, if $gf$ is an inflation, then so is $f$.

  (2) For any pair of morphisms $f:X\rightarrow Y$ and $g:Y\rightarrow Z$ in $\mathscr{C}$, if $gf$ is a deflation, then so is $g$.
\end{condition}

\begin{definition}
  \cite[Definition 2.9]{16} A sequence $A\stackrel{x}{\longrightarrow}B\stackrel{y}{\longrightarrow}C$ is said to be {\em right exact $\mathbb{E}$-triangle} if there exists an $\mathbb{E}$-triangle $K\stackrel{h_2}{\longrightarrow}B\stackrel{y}{\longrightarrow}C\stackrel{\delta}{\dashrightarrow}$ and a deflation $h_1:A\rightarrow K$ which is compatible, such that $x=h_2h_1$. Dually, we can define {\em left exact $\mathbb{E}$-triangles}.
\end{definition}

A morphism $f$ in $\mathscr{C}$ is called {\em compatible}, if ``$f$ is both an inflation and a deflation'' implies $f$ is an isomorphism.

\begin{lemma}\label{2.8}
  \cite[Lemma 2.10]{16} Let $\eta:A\stackrel{x}{\longrightarrow}B\stackrel{y}{\longrightarrow}C$ be a right exact $\mathbb{E}$-triangle in $\mathscr{C}$. If x is an inflation, then $\eta$ is a conflation. Dually, we have the similar result on left exact $\mathbb{E}$-triangles.
\end{lemma}

\begin{corollary}
  \cite[Remark 2.11]{16} A sequence $\eta:A\stackrel{x}{\longrightarrow}B\stackrel{y}{\longrightarrow}C$ is both right exact and left exact if and only if $\eta$ is a conflation.
\end{corollary}

For any object $C\in\mathscr{C}$ , there exist $\mathbb{E}$-triangles $$A\stackrel{x}{\longrightarrow}P\stackrel{y}{\longrightarrow}C\stackrel{\delta}{\dashrightarrow} \;{\rm and}\; C\stackrel{x}{\longrightarrow}I\stackrel{y}{\longrightarrow}A'\stackrel{\delta}{\dashrightarrow}$$
with $P\in\mathcal{P}(\mathscr{C})$ and $I\in\mathcal{I}(\mathscr{C})$. In this case, $A$ is called the {\em syzygy} and $A'$ is called the {\em cosyszygy} of $C$, which are denoted by $\Omega(C)$ and $\Sigma(C)$, respectively.

For any subcategory $\mathcal{U}$ of $\mathscr{C}$, put $\Omega^0\mathcal{U}=\mathcal{U}$, and for $i>0$ we define $\Omega^i\mathcal{U}$ inductively by $\Omega^i(\mathcal{U})=\Omega(\Omega^{i-1}\mathcal{U})$, i.e. the subcategory consisting of syzygies of objects in $\Omega^{i-1}\mathcal{U}$. We call $\Omega^i\mathcal{U}$ the {\em i}-th syzygy of $\mathcal{U}$. Dually, we can define the {\em i}-th cosyzygy of $\mathcal{U}$ denoted by $\Sigma^i\mathcal{U}$ for $i\geqslant0$.

By \cite[Lemma 5.1]{17}, the higher extension group is defined as
$$\mathbb{E}^{i+1}(X,Y)\cong\mathbb{E}(X,\Sigma^iY)\cong\mathbb{E}(\Omega^iX,Y)$$
for any $X,Y\in\mathscr{C}$ and $i\geqslant0$.

\begin{remark}
  For any subcategory $\mathcal{U}$ of $\mathscr{C}$, we define $$\mathcal{U}^{\bot_n}=\{X\in\mathscr{C}~|~\mathbb{E}^n(U,X)=0~\text{for any}~ U\in\mathcal{U}\}$$ for any $n\geqslant1$. Dually, we define $^{\bot_n}\mathcal{U}$. In particular, $$\mathcal{U}^{\bot_0}=\{X\in\mathscr{C}~|~\mathscr{C}(U,X)=0~\text{for any}~U\in\mathcal{U}\},$$ $$^{\bot_0}\mathcal{U}=\{X\in\mathscr{C}~|~\mathscr{C}(X,U)=0~\text{for any}~U\in\mathcal{U}\}.$$
\end{remark}

The following result is well-known in \cite{17} and we will use it frequently in the following.

\begin{proposition}\label{2.13}
   \cite[Proposition 5.2]{17} Let $A\stackrel{x}{\longrightarrow}B\stackrel{y}{\longrightarrow}C\stackrel{\delta}{\dashrightarrow}$ be an $\mathbb{E}$-triangle. There exist long exact sequences $$\ldots\longrightarrow \mathbb{E}^i(X,A)\longrightarrow\mathbb{E}^i(X,B)\longrightarrow\mathbb{E}^i(X,C)$$
  $$\,\,\,\longrightarrow\mathbb{E}^{i+1}(X,A)\longrightarrow\mathbb{E}^{i+1}(X,B)\longrightarrow\ldots $$ and $$\ldots\longrightarrow \mathbb{E}^i(C,X)\longrightarrow\mathbb{E}^i(B,X)\longrightarrow\mathbb{E}^i(A,X)$$
  $$\,\,\,\longrightarrow\mathbb{E}^{i+1}(C,X)\longrightarrow\mathbb{E}^{i+1}(B,X)\longrightarrow\ldots $$ for any objects $X\in\mathscr{C}$ and $i\geqslant0$.
\end{proposition}

Let's recall some definitions in \cite{7}. An {\em $\mathbb{E}$-sequence} in $\mathscr{C}$ is defined as a sequence $$\ldots\longrightarrow X_{n-1}\stackrel{d_{n-1}}{\longrightarrow}X_{n}\stackrel{d_n}{\longrightarrow}X_{n+1}\longrightarrow\ldots$$ such that for any $n$, there are $\mathbb{E}$-triangles $$K_{n}\stackrel{g_n}{\longrightarrow}X_n\stackrel{f_n}{\longrightarrow}K_{n+1}\stackrel{\delta_n}{\dashrightarrow}$$ in $\mathscr{C}$ and the differential $d_n=g_{n+1}f_n$.

For any $A\in \mathscr{C}$, $A$ admits an $\mathbb{E}$-sequence  $$\ldots\longrightarrow P_{n}\stackrel{d_{n}}{\longrightarrow}\ldots\longrightarrow P_{1}\stackrel{d_1}{\longrightarrow}P_{0}\stackrel{d_0}{\longrightarrow}A$$ with $P_i\in\mathcal{P}(\mathscr{C})$, which is called a {\em projective resolution} of $A$. We define the {\em projective dimension} of $A$, denoted by ${\rm pd}(A)$, to be the minimal length of all projective resolutions of $A$. If all projective resolutions of $A$ are of infinite length, we set ${\rm pd}(A)=+\infty$. For any subcategory $\mathcal{U}\subseteq\mathscr{C}$, we define ${\rm pd}(\mathcal{U})={\rm sup}\{{\rm pd}(U)\;|\;U\in\mathcal{U}\}$. Dually, we can define the {\em injective resolution} and the {\em injective dimension} of $A\in \mathscr{C}$ which is denoted by ${\rm id}(A)$. Similarly, we can define ${\rm id}(\mathcal{U})$.

\begin{lemma}\label{2.14}
  \cite[Lemma 3]{7} For any $A\in \mathscr{C}$, the following statements are equivalent:

  \quad{\rm (1)} ${\rm pd}(A)\leqslant n$;

  \quad{\rm (2)} $\mathbb{E}^{n+1}(A,X)$ for any $X\in\mathscr{C}$;

  \quad{\rm (3)} $\mathbb{E}^{n+i}(A,X)$ for any $X\in\mathscr{C}$ and $i\geqslant 1$.
\end{lemma}

The dual results also hold for id($A$).

\section{Tilting subcategories and torsion theory}

The tilting subcategories of extriangulated categories were introduced in \cite{7}. In this section, we give another definition, which is much more common, and prove the two definitions coincide with each other. As a generalization, we introduce the notion of support $\tau$-tilting subcategories in extriangulated categories. Besides, we define $\tau$-cotorsion pairs in extriangulated categories and give some propositions.

We still assume that $\mathscr{C}$ is an extriangulated category satisfying Condition \ref{2.6} in what follows.

\subsection{Tilting subcategories in extriangulated categories}
In this section, we assume all subcategories of $\mathscr{C}$ closed under direct summands and finite direct sum.

\begin{definition}\label{3.1}
  \cite[Definition 7]{7} Let $\mathcal{T}$ be a subcategory of $\mathscr{C}$. Then we call $\mathcal{T}$ a {\em tilting subcategory} of $\mathscr{C}$ if it satisfies the following conditions:

  (T1) ${\rm pd}(\mathcal{T})\leqslant 1$.

  (T2) $\mathcal{T}$ is a generator of $\mathcal{T}^{\bot_1}$.
\end{definition}

Note that a subcategory $\mathcal{U}\subseteq\mathscr{C}$ is a generator for another subcategory $\mathcal{V}\subseteq\mathscr{C}$ if $\mathcal{U}\subseteq\mathcal{V}$ and for any $V\in\mathcal{V}$, there is a conflation $V'\stackrel{x}{\longrightarrow}U\stackrel{y}{\longrightarrow}V
$ with $U\in\mathcal{U},V'\in\mathcal{V}$.
\vskip 10pt
Let us give another definition of tilting subcategories, which is more common.

\begin{definition}\label{3.2}
  Let $\mathcal{T}$ be a subcategory of $\mathscr{C}$. Then $\mathcal{T}$ is called a tilting subcategory of $\mathscr{C}$ if it satisfies the following conditions:

  (T1) ${\rm pd}(\mathcal{T})\leqslant 1$.

  (T2$'$) $\mathbb{E}(\mathcal{T},\mathcal{T})=0$.

  (T3) For any $P\in\mathcal{P}(\mathscr{C})$, there is a conflation

   $$P\longrightarrow T_0\longrightarrow T_1$$ \quad\quad\quad with $T_0,T_1\in\mathcal{T}$.
\end{definition}

\begin{theorem}\label{3.6}
  The definitions of tilting subcategories given in Definition \ref{3.1} and Definition \ref{3.2} are equivalent.
\end{theorem}

\begin{proof}
  By \cite[Remark 2]{7}, the conditions in Definition \ref{3.1} clearly induce the conditions in Definition \ref{3.2}.

  On the contrary, it suffices to prove (T2). For any $F\in{\rm Defl}\mathcal{T}$, there is a conflation $$A \longrightarrow T_F \longrightarrow F$$ with $T_F\in\mathcal{T}$ and $A\in\mathscr{C}$. Then we get a long exact sequence
  $$\ldots\longrightarrow\mathbb{E}(T,T_F)\longrightarrow \mathbb{E}(T,F)\longrightarrow\mathbb{E}^2(T,A){\longrightarrow}\ldots$$
  for any $T\in \mathcal{T}$ by Proposition \ref{2.13}. Since ${\rm pd}(\mathcal{T})\leqslant 1$, we have $\mathbb{E}^{2}(T,A)=0$. Moreover, by
 condition (T2$'$), we have $\mathbb{E}(T,T_F)=0$. Hence, we have $\mathbb{E}(\mathcal{T},{\rm Defl}\mathcal{T})=0$. Since $\mathscr{C}$ has enough projectives, there exist conflations
  $$
  Y\longrightarrow P_0\longrightarrow X,\;Z\longrightarrow P_1\longrightarrow Y
  $$
  with $P_0,P_1\in\mathcal{P}(\mathscr{C})$ for any $X\in \mathcal{T}^{\bot_1}$. Then there exists a commutative diagram of conflations
  $$
  \xymatrix{Z \ar[r] \ar@{=}[d]& P_1\ar[r]\ar[d]&Y\ar[d]\\
  Z\ar[r]&T^0\ar[r]\ar[d]&E\ar[d]\\
  &T^1\ar@{=}[r]&T^1}
  $$
  with $T^0,T^1\in \mathcal{T}$. By the duality of \cite[Proposition 3.15]{8}, we have a commutative diagram of conflations
  $$\xymatrix{Y \ar[r] \ar[d]& P_0\ar[r]\ar[d]&X\ar@{=}[d]\\
  E\ar[r]\ar[d]&M\ar[r]\ar[d]&X\\
  T^1\ar@{=}[r]&T^1&.}
  $$
  Since $E\in {\rm Defl}\mathcal{T}\subseteq \mathcal{T}^{\bot_1}$ and $X\in \mathcal{T}^{\bot_1}$, we have $M\in \mathcal{T}^{\bot_1}$ by Proposition \ref{2.13}. There exists a commutative diagram of conflations
  $$\xymatrix{P_0 \ar[r] \ar[d]& M\ar[r]\ar[d]&T^1\ar@{=}[d]\\
  T'^0\ar[r]\ar[d]&T'\ar[r]\ar[d]&T^1\\
  T'^1\ar@{=}[r]&T'^1}
  $$
  with $T'^0,T'^1\in \mathcal{T}$. The second row is split, since $\mathbb{E}(\mathcal{T},\mathcal{T})=0$. Then $T'\in\mathcal{T}$. We have a commutative diagram of conflations
  $$\xymatrix{E \ar[r] \ar@{=}[d]& M\ar[r]\ar[d]&X\ar[d]\\
  E\ar[r]&T'\ar[r]\ar[d]&L\ar[d]\\
  &T'^1\ar@{=}[r]&T'^1.}
  $$
  The third column is split, since $X\in \mathcal{T}^{\bot_1}$. Noting that $T'^1\longrightarrow
  L\longrightarrow X$ is also a conflation, we have a commutative diagram of conflations
  $$\xymatrix{E \ar[r] \ar@{=}[d]& M'\ar[r]\ar[d]&T'^1\ar[d]\\
  E\ar[r]&T'\ar[r]\ar[d]&L\ar[d]\\
  &X\ar@{=}[r]&X.}
  $$
  Clearly, $M'\in \mathcal{T}^{\bot_1}$. So $\mathcal{T}$ is a generator of $\mathcal{T}^{\bot_1}$. Hence, we complete the proof.
\end{proof}

Suggested by the referee, one may define another definition for tilting subcategory, as the following result shows.

\begin{theorem}
  (A criterion for tilting subcategories)Let $\mathscr{C}$ be locally small and $\mathcal{T}$ a subcategory of $\mathscr{C}$. Then $\mathcal{T}$ is a tilting subcategory if and only if ${\rm Defl}\mathcal{T}=\mathcal{T}^{\bot_1}$.
\end{theorem}

\begin{proof}
  The necessity is clear. Indeed, the arguments in the proof of Theorem \ref{3.6} show that if $\mathcal{T}$ is tilting, then ${\rm Defl}\mathcal{T}=\mathcal{T}^{\bot_1}$. It suffices to prove the sufficiency. Since $\mathscr{C}$ has enough injectives, there is a conflation $$X\longrightarrow
  I\longrightarrow Y$$ with $I\in\mathcal{I}(\mathscr{C})$ for any $X\in\mathscr{C}$. Now since $I\in\mathcal{T}^{\bot_1}={\rm Defl}\mathcal{T}$, we get $Y\in \mathcal{T}^{\bot_1}$. By Proposition \ref{2.13}, we have $\mathbb{E}^{2}(T,X)=\mathbb{E}(T,Y)=0$ for any $T\in\mathcal{T}$. Hence, $\mathcal{T}$ satisfies (T1).

  For any $D\in\mathcal{T}^{\bot_1}={\rm Defl}\mathcal{T}$, there is a deflation $g:T_1\longrightarrow D$ with $T_1\in \mathcal{T}$. Consider the set of morphisms $\Lambda=\{f:T^f\longrightarrow D~|~T^f\in \mathcal{T} \}$. Take $T_1=\bigoplus\limits_{f\in\Lambda}T^f\in\mathcal{T}$ and we have $G:T_1\longrightarrow D$ naturally. Now since $g\in \Lambda$ and it clearly factors through the $G$, $G:T_1 \longrightarrow D$ is a deflation by Condition WIC. Therefore, we get a conflation $$D_1\longrightarrow T_1\stackrel{G}{\longrightarrow} D.$$ By Proposition \ref{2.13}, for any $T\in\mathcal{T}$, we have a long exact sequence $$\ldots\longrightarrow \mathscr{C}(T,T_1) \stackrel{G^*}{\longrightarrow} \mathscr{C}(T,D)\longrightarrow \mathbb{E}(T,D_1)\longrightarrow\mathbb{E}(T,T_1){\longrightarrow}\ldots$$ where $G^*$ is an epimorphism by the construction of $\Lambda$. Moreover, $\mathbb{E}(T,T_1)=0$ since $T_1\in\mathcal{T}\subseteq{\rm Defl}\mathcal{T}=\mathcal{T}^{\bot_1}$. Now since $\mathbb{E}(T,D_1)=0$, we get $\mathcal{T}$ satisfies (T2). Hence, we complete the proof.
\end{proof}

Then we will introduce the notion of support $\tau$-tilting subcategories and give an equivalence between tilting subcategories and support $\tau$-tilting subcategories.

\begin{definition}
  Let $\mathcal{T}$ be a subcategory of $\mathscr{C}$. Then $\mathcal{T}$ is called a {\em support $\tau$-tilting subcategory} of $\mathscr{C}$ if it satisfies the following conditions:

  (ST1) $\mathcal{T}$ is a generator of ${\rm Defl}\mathcal{T}$.

  (ST2) $\mathbb{E}(\mathcal{T},{\rm Defl}\mathcal{T})=0$.

  (ST3) For any $P\in\mathcal{P}(\mathscr{C})$, there is a right exact $\mathbb{E}$-triangle
  $$P\stackrel{f}{\longrightarrow}T^0 \longrightarrow T^1,$$ \quad\;\;\;\, where $T^0$,$T^1\in\mathcal{T}$ and $f$ is a left $\mathcal{T}$-approximation of $P$.
\end{definition}

\begin{lemma}\label{3.4}
   Let $\mathcal{T}$ be a subcategory of $\mathscr{C}$. Then $\mathcal{T}$ is a tilting subcategory if and only if $\mathcal{T}$ is a support $\tau$-tilting subcategory and there is an inflation $f$ satisfying {\rm(ST3)} for each $P\in\mathcal{P}(\mathscr{C})$.
\end{lemma}

\begin{proof}
  For the necessity, it suffices to show such $f$ is a left $\mathcal{T}$-approximation of $P$ by Definition \ref{3.1} and the proof of Theorem \ref{3.6}. There is a conflation $P\stackrel{f}{\longrightarrow}T^0\longrightarrow T^1$ by (T3). Then we get another long exact sequence
  $$\mathscr{C}(T^1,T)\longrightarrow\mathscr{C}(T^0,T)\stackrel{\mathscr{C}(f,T)}{\longrightarrow} \mathscr{C}(P,T)\longrightarrow\mathbb{E}(T^1,T)\longrightarrow\ldots$$ for any $T\in \mathcal{T}$. Since $\mathbb{E}(\mathcal{T},\mathcal{T})=0$, we get $\mathscr{C}(f,T)$ is surjective. This implies $f$ is a left $\mathcal{T}$-approximation of $P$.

  For the sufficiency, we only need to show ${\rm pd}(\mathcal{T})\leqslant 1$. It suffices to prove $\mathbb{E}^2(T,A)=0$ for any $T\in\mathcal{T}$ and $A\in\mathscr{C}$ by Lemma \ref{2.14}. Consider the conflation
  $$A\longrightarrow I \longrightarrow \Sigma A $$
  with $I\in\mathcal{I}(\mathscr{C})$. It suffices to show $\mathbb{E}(T,\Sigma A)=0$. Since $\mathscr{C}$ has enough projectives, there exists a conflation $$X\longrightarrow P\longrightarrow I$$ with $P\in\mathcal{P}(\mathscr{C})$.  By assumptions and Lemma \ref{2.8}, we have a conflation $$P\longrightarrow T^0 \longrightarrow T^1.$$
  Then there is a commutative diagram of conflations
  $$\xymatrix{X \ar[r] \ar@{=}[d]& P\ar[r]\ar[d]&I\ar[d]\\
  X\ar[r]&T^0\ar[r]\ar[d]&M\ar[d]^a\\
  &T^1\ar@{=}[r]&T^1}$$ in which $I\longrightarrow M\stackrel{a}\longrightarrow T^1$ is a split conflation, since $I$ is injective. Hence, there exists a morphism $b:T^1\longrightarrow M$ such that $ab={\rm Id}_{T^1}$. Then there is a commutative diagram of conflations
  $$\xymatrix{A \ar[r] \ar@{=}[d]& I\ar[r]\ar[d]&\Sigma A\ar[d]\\
  A\ar[r]&M\ar[r]^c\ar[d]^a&E\ar[d]^d\\
  &T^1\ar@{=}[r]&T^1}
  $$
  with $M,E\in{\rm Defl}\mathcal{T}$. By Propositon \ref{2.13}, there is a long exact sequence
  $$\ldots\longrightarrow\mathscr{C}(T,E)\stackrel{\mathscr{C}(T,d)}{\longrightarrow}\mathscr{C}(T,T^1) \longrightarrow \mathbb{E}(T,\Sigma A)\longrightarrow\mathbb{E}(T,E)=0
  $$
  For any $\phi\in\mathscr{C}(T,T^1)$, $cb\phi\in\mathscr{C}(T,E)$ and $dcb\phi=ab\phi=\phi$. Since $\mathscr{C}(T,d)$ is a surjection, we have $\mathbb{E}(T,\Sigma A)=0$. Hence, we complete the proof.
\end{proof}

Lemma \ref{3.4} is a generalization of \cite[Lemma 3.3]{15}, which gives the same result in exact categories. However, \cite[Lemma 3.3]{15} uses some tools of higher extensions in exact categories, which may not work in extriangulated categories. Hence, we just use the definition of higher extensions in extriangulated categories to prove. The lemma below is pivotal in the proof of the main result.

\begin{lemma}\label{3.5}
  Let $\mathcal{T}$ be a support $\tau$-tilting subcategory of $\mathscr{C}$. Then ${\rm Defl}\mathcal{T}$ is closed under extensions and ${^{\bot_1}{\rm Defl}\mathcal{T}}\cap{\rm Defl}\mathcal{T}=\mathcal{T}$.
\end{lemma}

\begin{proof}
  Let $D^1\stackrel{f}{\longrightarrow} A\stackrel{g}{\longrightarrow} D^2\stackrel{\delta}{\dashrightarrow}$ be an $\mathbb{E}$-triangle with $D^1,D^2\in{\rm Defl}\mathcal{T}$. Then we have two deflations $T^1\stackrel{f'}{\longrightarrow}D^1$ and $T^2\stackrel{g'}{\longrightarrow}D^2$. Since $\mathbb{E}(\mathcal{T},{\rm Defl}\mathcal{T})=0$, we have a commutative diagram
  $$\xymatrix{T^1 \ar[r]^{\tiny{\begin{pmatrix}1\\0\end{pmatrix}}\quad\;} \ar[d]^{f'}& T^1\oplus T^2\ar[r]^{\quad\tiny{\begin{pmatrix}0 & 1\end{pmatrix}}}\ar[d]^{h'}& T^2\ar[d]^{g'}\\
  D^1\ar[r]^{f}&A\ar[r]^{g}&D^2.}$$ By \cite[Lemma 4.15]{18}, we get $h'$ is a deflation. It is clear that $A\in {\rm Defl}\mathcal{T}$.

  Clearly,  $\mathcal{T}\subseteq{^{\bot_1}{\rm Defl}\mathcal{T}}\cap{\rm Defl}\mathcal{T}$. On the contrary, for any $D\in{^{\bot_1}{\rm Defl}\mathcal{T}}\cap{\rm Defl}\mathcal{T}$, there is a split conflation $$D'\longrightarrow T \longrightarrow  D$$ with $T\in \mathcal{T},D'\in {\rm Defl}\mathcal{T}$, since $\mathcal{T}$ is a generator of ${\rm Defl}\mathcal{T}$. Hence, $D\in \mathcal{T}$ and the proof is finished.
\end{proof}

\subsection{Cotorsion pairs in extriangulated categories}
We recall the notion of cotorsion pairs in extriangulated categories and define $\tau$-cotorsion pairs.

\begin{definition}
  \cite[Definition 4.1]{8} Let $\mathcal{U},\mathcal{V}\subseteq\mathscr{C}$ be a pair of subcategories, which are closed under direct summands. The pair $(\mathcal{U},\mathcal{V})$  is called a {\em cotorsion pair} in $\mathscr{C}$ if it satisfies the following conditions:

  (1) $\mathbb{E}(\mathcal{U},\mathcal{V})=0$.

  (2) For any $C\in \mathscr{C}$, there is a conflation
  $$V^C\longrightarrow U^C\longrightarrow C$$
  \quad\quad\,\, satisfying $U^C\in \mathcal{U},V^C\in \mathcal{V}$.

  (3) For any $C\in \mathscr{C}$, there is a conflation
  $$C\longrightarrow V_C\longrightarrow U_C$$
  \quad\quad\,\, satisfying $U_C\in \mathcal{U},V_C\in \mathcal{V}$.
\end{definition}




\begin{definition}
   Let $\mathcal{U},\mathcal{V}\subseteq\mathscr{C}$ be a pair of subcategories. The pair $(\mathcal{U},\mathcal{V})$  is called a {\em $\tau$-cotorsion pair} in $\mathscr{C}$ if it satisfies the following conditions:

   (1) $\mathcal{U}={^{\bot_1}\mathcal{V}}$.

   (2) For any $P\in \mathcal{P}(\mathscr{C})$, there is a right exact $\mathbb{E}$-triangle
  $$P\stackrel{f}{\longrightarrow} V\longrightarrow U$$
  \quad\quad\,\, where $V\in \mathcal{U}\cap\mathcal{V},U\in \mathcal{U}$ and $f$ is a left $\mathcal{V}$-approximation of $P$.
\end{definition}

\begin{lemma}\label{3.8}
  Each cotorsion pair in $\mathscr{C}$ is a $\tau$-cotorsion pair.
\end{lemma}

\begin{proof}
  Let $(\mathcal{U},\mathcal{V})$  be a cotorsion pair. Then we get $\mathcal{U}= {^{\bot_1}\mathcal{V}}$ by \cite[Remark 4.4]{8}. For any $P\in \mathcal{P}(\mathscr{C})$, there is a conflation
  $$V^P\longrightarrow U^P\longrightarrow P$$
  with $U^P\in \mathcal{U},V^P\in \mathcal{V}$. Since $P$ is a direct summand of $U^P$, we get $P\in \mathcal{U}$. Note that we have another conflation
  $$P\stackrel{f}{\longrightarrow} V_P\longrightarrow U_P$$
  satisfying $U_P\in \mathcal{U},V_P\in \mathcal{V}$. Now since $\mathcal{U}$ is closed under extensions by \cite[Remark 4.6]{8}, we have $V_P\in \mathcal{U}\cap\mathcal{V}$. Since $\mathbb{E}(\mathcal{U},\mathcal{V})=0$, $f$ is a left $\mathcal{V}$-approximation of $P$. Hence, the proof is finished.
\end{proof}

\begin{lemma}\label{3.9}
  Let $(\mathcal{U},\mathcal{V})$ be a $\tau$-cotorsion pair of $\mathscr{C}$. Then $\mathcal{U}$ is closed under extensions and direct summands.
\end{lemma}

\begin{proof}
  Let $U^1\stackrel{x}{\longrightarrow}U\stackrel{y}{\longrightarrow}U^2\stackrel{\delta}{\dashrightarrow}$ be an $\mathbb{E}$-{\em triangle} with $U^1,U^2\in\mathcal{U}$. Using Proposition \ref{2.13} we get $U\in\mathcal{U}$. Hence, $\mathcal{U}$ is closed under extensions.

  Let $U\in\mathcal{U}$ with $U=U^1\oplus U^2$. For any $V\in\mathcal{V}$, $0=\mathbb{E}(U,V)=\mathbb{E}(U^1,V)\oplus \mathbb{E}(U^2,V)$. Then $\mathbb{E}(U^1,V)=\mathbb{E}(U^2,V)=0$. Hence, $\mathcal{U}$ is closed under direct summands.
\end{proof}


\subsection{Main results}
In this section, we provide a bijection between support $\tau$-tilting subcategories and certain $\tau$-cotorsion pairs in an extriangulated category.

A subcategory $\mathcal{V}$ of $\mathscr{C}$ is said to be a {\em torsion class}, if $\mathcal{V}$ is closed under extensions and deflations.

\begin{theorem}\label{3.11}
  There are mutually inverse bijections:
 $$\xymatrix{\{\text{support~$\tau$-tilting~subcategories}\}\ar@{<->}[r]&\{\text{$\tau$-cotorsion~ pair}~(\mathcal{U},\mathcal{V})~|~\mathcal{V}~\text{is~a~torsion~class}\}}$$
  $$\xymatrix{\,\,\mathcal{T}\ar@{|->}[r]&({^{\bot_1}{\rm Defl}\mathcal{T}},{\rm Defl}\mathcal{T})}\,\,$$
  $$\xymatrix{\mathcal{U}\cap\mathcal{V}&(\mathcal{U},\mathcal{V})\ar@{|->}[l]\, .\quad\quad\quad\quad\quad\quad}$$

\end{theorem}

\begin{proof}
   Let $\mathcal{T}$ be a support $\tau$-tilting subcategory of $\mathscr{C}$. Then, for any $P\in \mathcal{P}(\mathscr{C})$, there is a right exact $\mathbb{E}$-triangle
  $$P\stackrel{f}{\longrightarrow}T^0\stackrel{g}{\longrightarrow}T^1,$$
  with $T^0$,$T^1\in\mathcal{T}={^{\bot_1}{\rm Defl}\mathcal{T}}\cap{\rm Defl}\mathcal{T}$. Moreover, $f$ is a left $\mathcal{T}$-approximation of $P$. In fact, $f$ is also a left ${\rm Defl}\mathcal{T}$-approximation of $P$. Hence, $({^{\bot_1}{\rm Defl}\mathcal{T}},{\rm Defl}\mathcal{T})$ is a $\tau$-cotorsion pair with ${\rm Defl}\mathcal{T}$ a torsion class by Lemma \ref{3.5}.

  Let $(\mathcal{U},\mathcal{V})$ be a $\tau$-cotorsion pair with $\mathcal{V}$ a torsion class. Note that $\mathcal{U},\mathcal{V}$ are closed under extensions and direct summands by Lemma \ref{3.9}. Then so is $\mathcal{U}\cap\mathcal{V}$. We get $\mathbb{E}(\mathcal{U}\cap\mathcal{V},{\rm Defl}(\mathcal{U}\cap\mathcal{V}))=0$, since $\mathcal{U}={^{\bot_1}\mathcal{V}}$ and $\mathcal{V}$ is closed under deflations. For any $P\in \mathcal{P}(\mathscr{C})$, there is a conflation
  $$P\stackrel{f}{\longrightarrow} V\longrightarrow U$$
  where $V\in \mathcal{U}\cap\mathcal{V},U\in \mathcal{U}$ and $f$ is a left $\mathcal{V}$-approximation of $P$. Since $\mathcal{V}$ is a torsion class, we get $U\in \mathcal{U}\cap\mathcal{V}$. For any $D\in {\rm Defl}(\mathcal{U}\cap\mathcal{V})$, there is a conflation
  $$X\longrightarrow V_D\longrightarrow D$$
  with $V_D\in \mathcal{U}\cap\mathcal{V}$. Since $\mathscr{C}$ has enough projectives, there is a conflation
  $$Y\longrightarrow P'\longrightarrow X$$
  with $P'\in \mathcal{P}(\mathscr{C})$. Now since $(\mathcal{U},\mathcal{V})$ is a $\tau$-cotorsion pair, for $P'\in \mathcal{P}(\mathscr{C})$, there is a conflation $$P'{\longrightarrow} V_{P'}\longrightarrow U_{P'}$$ with $V_{P'}\in \mathcal{U}\cap\mathcal{V}$ and $U_{P'}\in \mathcal{U}$. Since $\mathcal{V}$ is a torsion class, we have $U_{P'}\in \mathcal{U}\cap\mathcal{V}$. Consider the commutative diagram of conflations
  $$\xymatrix{Y \ar[r] \ar@{=}[d]& P'\ar[r]\ar[d]&X\ar[d]\\
  Y\ar[r]&V_{P'}\ar[r]\ar[d]&M\ar[d]\\
  &U_{P'}\ar@{=}[r]&U_{P'}.}
  $$
  Clearly, $M\in{\rm Defl}(\mathcal{U}\cap\mathcal{V})$. There is a commutative diagram of conflations
  $$\xymatrix{X\ar[r] \ar[d]& V_D\ar[r]\ar[d]&D\ar@{=}[d]\\
  M\ar[r]\ar[d]&E\ar[r]\ar[d]&D\\
  U_{P'}\ar@{=}[r]&U_{P'}&.}
  $$
  Since $\mathcal{U}\cap\mathcal{V}$ is closed under extensions, we get $E\in \mathcal{U}\cap\mathcal{V}$. Thus the second row of the diagram induces that $\mathcal{U}\cap\mathcal{V}$ is a generator of ${\rm Defl}(\mathcal{U}\cap\mathcal{V})$. Hence, $\mathcal{U}\cap\mathcal{V}$ is a support $\tau$-tilting subcategory of $\mathscr{C}$.

  To prove the two maps are mutually inverse bijections, it suffices to show $\mathcal{V}={\rm Defl}(\mathcal{U}\cap\mathcal{V})$ by Lemma \ref{3.5}. Clearly, ${\rm Defl}(\mathcal{U}\cap\mathcal{V})\subseteq{\rm Defl}\mathcal{V}=\mathcal{V}$. Conversely, for any $X\in\mathcal{V}$, there is a deflation $g:P\longrightarrow X$ with $P\in \mathcal{P}(\mathscr{C})$. By the definition of $\tau$-cotorsion pairs, we have a commutative diagram
  $$\xymatrix{P\ar[r]^f\ar[d]^g&V\ar[r]\ar@{-->}[ld]^h&U\\
  X&&.}$$
  with $V\in \mathcal{U}\cap\mathcal{V},U\in \mathcal{U}$ and $f$ is a left $\mathcal{V}$-approximation of $P$. There exists a morphism $h:V\longrightarrow X$ such that $g=hf$. Now since $h$ is a deflation by Condition WIC, we get $X\in {\rm Defl}(\mathcal{U}\cap\mathcal{V})$. Therefore, we complete the proof.
\end{proof}

\begin{corollary}
  The bijection between support $\tau$-tilting subcategories and certain $\tau$-cotorsion pairs induces a bijection between tilting subcategories and certain cotorsion pairs.
\end{corollary}

\begin{proof}
  Let $\mathcal{T}$ be a support tilting subcategory of $\mathscr{C}$. Similar to Lemma \ref{3.9}, ${^{\bot_1}{\rm Defl}\mathcal{T}}$ is closed under direct summands. By definition, it is clear that ${\rm Defl}\mathcal{T}$ is also closed under direct summands. For any $X\in \mathscr{C}$, there is a conflation $$Y\longrightarrow P\longrightarrow X$$ with $P\in \mathcal{P}(\mathscr{C})$ and a conflation
  $$P\longrightarrow T^0\longrightarrow T^1$$ with $T^0$,$T^1\in\mathcal{T}\subseteq{\rm Defl}\mathcal{T}$. Consider the commutative diagram of conflations
  \begin{equation}\label{*}
  \xymatrix{Y\ar[r] \ar@{=}[d]& P\ar[r]\ar[d]&X\ar[d]\\
  Y\ar[r]&T^0\ar[r]\ar[d]&W\ar[d]\\
  &T^1\ar@{=}[r]&T^1}
  \end{equation}
  where $X\longrightarrow W\longrightarrow T^1$ is a conflation with $W\in{\rm Defl}\mathcal{T}$ and $T^1\in\mathcal{T}\subseteq {^{\bot_1}{\rm Defl}\mathcal{T}}$. There exists a conflation $$Z\longrightarrow P'\longrightarrow Y$$ with $P'\in \mathcal{P}(\mathscr{C})$. Then we have a conflation
  $$P'\longrightarrow T'^0 \longrightarrow T'^1$$ with $T'^0$,$T'^1\in\mathcal{T}\subseteq{\rm Defl}\mathcal{T}$. Similar to (\ref{*}), we get a conflation $Y\longrightarrow W'\longrightarrow T'^1$ with $W'\in{\rm Defl}\mathcal{T}$ and $T'^1\in {^{\bot_1}{\rm Defl}\mathcal{T}}$. Consider the commutative diagram of conflations
  $$\xymatrix{Y \ar[r] \ar[d]& P\ar[r]\ar[d]&X\ar@{=}[d]\\
  W'\ar[r]\ar[d]&N\ar[r]\ar[d]&X\\
  T'^1\ar@{=}[r]&T'^1}$$
  where $W'\longrightarrow N\longrightarrow X$ is a conflation with $W'\in{\rm Defl}\mathcal{T}$ and $T'^1,P\in {^{\bot_1}{\rm Defl}\mathcal{T}}$. Clearly, $N\in{^{\bot_1}{\rm Defl}\mathcal{T}}$ by Proposition \ref{2.13}.
  So $({^{\bot_1}{\rm Defl}\mathcal{T}},{\rm Defl}\mathcal{T})$ is a cotorsion pair with ${\rm Defl}\mathcal{T}$ a torsion class by Lemma \ref{3.5}.

  Let $(\mathcal{U},\mathcal{V})$ be a cotorsion pair with $\mathcal{V}$ a torsion class, then $\mathcal{U}\cap\mathcal{V}$ is a support $\tau$-tilting subcategory of $\mathscr{C}$ by Lemma \ref{3.8} and Theorem \ref{3.11}. Then $\mathcal{U}\cap\mathcal{V}$ is a tilting subcategory of $\mathscr{C}$ by Lemma \ref{3.4} and the proof of Lemma \ref{3.8}.
\end{proof}


\end{document}